\documentclass[12pt,english]{article}

\usepackage[T1]{fontenc}
\usepackage[utf8]{inputenc}   
\usepackage{microtype}
\usepackage{comment}
\usepackage{tikz}
\usetikzlibrary{decorations.markings}
\usetikzlibrary{arrows.meta} 

\usepackage{graphicx} 
\usepackage{amsmath}
\usepackage{amsthm}
\usepackage{amssymb}
\usepackage{newtxtext}
\usepackage[cmintegrals,varg]{newtxmath} 
\theoremstyle{plain} 
\newtheorem{theorem}{Theorem}

\newtheorem{corollary}{Corollary}
\newtheorem{proposition}{Proposition}

\theoremstyle{remark} 

\usepackage[authoryear]{natbib}
\usepackage{tabularx}
\usepackage{siunitx}
\usepackage{booktabs}

\usepackage{geometry}
\usepackage{setspace}
\makeatother

\title{The Parabolic Mellin Transform: Gamma and Zeta Integral Representations}

\author{Peter Reinhard Hansen$^{\mathsection}$\thanks{Corresponding author. Email: hansen@unc.edu. A preliminary version of this manuscript was circulated as the arXiv preprint ``Complex Moments, Gamma and Riemann Zeta Functions unified by the Parabolic Mellin Transform.'' The present manuscript substantially revises the exposition, develops the contour-theoretic PMT formulation, and supersedes the earlier Gamma-function preprint \cite{HansenTong:2025ReciprocalGamma}.} \and Chen Tong$^{\ddagger}$
\\ \small $^{\mathsection}$Department of Economics, University of North Carolina at Chapel Hill
\\[-0.2cm] \small $^{\ddagger}$School of Economics, Xiamen University}

\date{\today}

\begin{document}

\maketitle

\begin{abstract}
We introduce the Parabolic Mellin Transform (PMT), defined by $\mathcal{P}_{\sigma}[f](z)=\int_{-\infty}^{\infty}w^{2z}f(w^2)dt$, where $w=\sigma+it$ and $\sigma>0$. Under the substitution $u=w^2$, the vertical line $\operatorname{Re}(w)=\sigma$ is mapped to the parabolic contour $C_\sigma$ in the $u$-plane. For the Gaussian kernel, the PMT yields $\int_{-\infty}^{\infty}w^{2z}e^{w^2}dt=\pi/\Gamma(\tfrac{1}{2}-z)=\cos(\pi z)\Gamma(z+\tfrac{1}{2})$, a parabolic-contour form of the classical Hankel representation for the reciprocal Gamma function. The advantage of this parametrization is that the contour integral becomes a Gaussian-damped vertical-line integral. 
We develop scaling, differentiation, and Dirichlet-composition identities for the PMT and use them to derive integral representations of the Hurwitz zeta, Riemann zeta, and Dirichlet eta functions. The framework provides a unified transform dictionary for Gamma-type and zeta-type special functions and yields equivalent reformulations of the Riemann hypothesis and the Lindel\"of hypothesis in terms of zeros and growth of parabolic-contour integrals.
\end{abstract}

\noindent%

\setlength{\abovedisplayskip}{8pt}
\setlength{\belowdisplayskip}{8pt}
\setlength{\abovedisplayshortskip}{5pt}
\setlength{\belowdisplayshortskip}{5pt}
\newpage{}
\section{Introduction}

We introduce the Parabolic Mellin Transform (PMT), a Mellin-type contour transform defined formally in Section \ref{sec:PMT}. It is based on the quadratic change of variables $u=w^2$, which sends the vertical line $w=\sigma+it$ to the parabolic contour $C_\sigma=\{(\sigma+it)^2:t\in\mathbb{R}\}$ in the $u$-plane. Since $\operatorname{Re}(w^2)=\sigma^2-t^2$, exponential kernels such as $e^u$ become Gaussian-damped vertical-line integrals.

The basic example is the Gaussian kernel. Writing $G(z)=\mathcal{P}_{\sigma}[e^u](z)$, we prove
$$
\frac{1}{\Gamma(s)}
=
\frac{1}{\pi}G(\tfrac{1}{2}-s),
\qquad
\Gamma(s)\sin(\pi s)
=
G(s-\tfrac{1}{2}).
$$
Thus the same entire function $G$ represents the Gamma function and its reciprocal at reflected arguments. Equivalently, $G(z)=\pi/\Gamma(\tfrac{1}{2}-z)=\cos(\pi z)\Gamma(z+\tfrac{1}{2})$, with removable singularities in the last expression.

We then use the scaling identity $G(z,\alpha)=\alpha^{-(z+1/2)}G(z)$ to connect the PMT with Dirichlet factors. This leads to parabolic integral representations of Hurwitz zeta, Riemann zeta, and Dirichlet eta functions. In particular, for suitable $\sigma$
$$
R(z,a)=G(z)\zeta(s,a),
\qquad
D(z,a)=G(z)\eta(s,a),\qquad s=z+\tfrac{1}{2},
$$
where $R$ and $D$ are the PMTs of $e^{au}/(1-e^u)$ and $e^{au}/(1+e^u)$, respectively.

These identities are closely related to classical Hankel-contour representations. The point is not to give a new analytic continuation of the Gamma or zeta functions, but to express the corresponding Hankel-type integrals through the parabolic parametrization $u=w^2$. This reorganization has three useful features: the quadratic parametrization supplies Gaussian endpoint decay on a vertical line; the Gaussian transform $G$ becomes a common normalizing factor for Gamma-type and zeta-type identities; and the shifted coordinate $z=s-\tfrac{1}{2}$ centers the critical strip at zero.

Finally, the factorization $R(z)=G(z)\zeta(z+\tfrac{1}{2})$ gives exact reformulations of standard critical-strip questions. Since $G$ has no zeros in $|\operatorname{Re}(z)|<\tfrac{1}{2}$, the zeros of $R$ in this strip coincide with the nontrivial zeros of $\zeta(z+\tfrac{1}{2})$. Similarly, the explicit identity $|G(i\tau)|=\sqrt{\pi\cosh(\pi\tau)}$ converts the Lindel\"of hypothesis into an equivalent growth condition for $R(i\tau)$. These are reformulations in PMT coordinates, not new zero-free regions or bounds.

\section{The Parabolic Mellin Transform}
\label{sec:PMT}

This section provides the formal definition of the transform and records the two analytic facts used throughout the paper: holomorphy under Gaussian decay and the dependence of the transform on the contour parameter $\sigma$. Throughout, $\log$ denotes the principal logarithm.

Fix $\sigma>0$ and write $w=\sigma+it$, $t\in\mathbb{R}$. For a function $f$ defined on the image of the map $w\mapsto w^2$, define
\begin{equation}
\mathcal{P}_{\sigma}[f](z)
=
\int_{-\infty}^{\infty}w^{2z}f(w^2)dt,
\qquad w=\sigma+it,
\label{eq:pmt-definition}
\end{equation}
whenever the integral exists. The branch of $w^{2z}$ is fixed by $\arg(w)\in(-\pi/2,\pi/2)$.

Let $C_{\sigma}=\{(\sigma+it)^2:t\in\mathbb{R}\}$ be the image of the vertical line $\operatorname{Re}(w)=\sigma$ under $u=w^2$. Thus $C_{\sigma}$ defines a left-opening parabolic contour with vertex at $\sigma^2$, illustrated in Figure \ref{fig:parabolic_contour}.
\begin{figure}[ht]
\centering
\begin{tikzpicture}[scale=0.9,
    > = stealth,
    midarrow/.style={
        decoration={markings, mark=at position 0.58 with {\arrow{>}}},
        postaction={decorate}
    }
]
    \draw[->] (-6.8,0) -- (2.6,0) node[below] {$\operatorname{Re}(u)$};
    \draw[->] (0,-4.6) -- (0,4.6) node[left] {$\operatorname{Im}(u)$};
    \draw[densely dotted, red, line width=1.0pt] (-6.5,0) -- (0,0);
    \node[below] at (-3.2,0) {\small $(-\infty,0]$};
    \draw[blue, line width=1.1pt, midarrow]
        plot[domain=-2.45:2.45, samples=120] ({1-\x*\x}, {2*\x});

    \fill (0,0) circle (1.2pt);
    \node[below right] at (0,0) {$0$};
    \fill[blue] (1,0) circle (1.4pt);
    \node[below right, blue] at (1,0) {$\sigma^2$};
    \node[blue] at (-1.9,2.1) {$C_\sigma$};
\end{tikzpicture}
\caption{The contour $C_\sigma=\{(\sigma+it)^2:t\in\mathbb{R}\}$ in the $u$-plane, oriented by increasing $t$, together with the branch cut $(-\infty,0]$.}
\label{fig:parabolic_contour}
\end{figure}

The principal branch of $u^{z-1/2}$ is well defined on $C_\sigma$. With $u=w^2$, we have $du=2iwdt$ and $w=u^{1/2}$ on $C_\sigma$, where the square root is the branch induced by $\operatorname{Re}(w)>0$. Hence
\begin{equation}
\mathcal{P}_{\sigma}[f](z)
=
\frac{1}{2i}\int_{C_{\sigma}}u^{z-1/2}f(u)du.
\label{eq:pmt-contour-form}
\end{equation}

Since $\operatorname{Re}(w^2)=\sigma^2-t^2$, kernels with exponential decay as $\operatorname{Re}(u)\to-\infty$ along $C_\sigma$ lead to Gaussian decay in the vertical-line parameter.

\begin{proposition}
\label{prop:pmt-entire}
Fix $\sigma>0$. Suppose that $f$ is holomorphic in an open neighborhood of $C_\sigma$ and that, for some constants $A,\eta>0$, $|f((\sigma+it)^2)|\leq A\exp(-\eta t^2)$, for $t\in\mathbb{R}$. 
Then $\mathcal{P}_\sigma[f](z)$ exists for every $z\in\mathbb{C}$ and defines an entire function of $z$. Moreover, for every integer $m\geq0$,
\begin{equation}
\frac{d^m}{dz^m}\mathcal{P}_\sigma[f](z)
=
\int_{-\infty}^{\infty}
(2\log w)^m w^{2z}f(w^2)dt,
\qquad w=\sigma+it.
\label{eq:pmt-derivatives}
\end{equation}
\end{proposition}

\begin{proof}
Let $K\subset\mathbb{C}$ be compact. For $z\in K$,
$$
|w^{2z}|
=
\exp\{2\operatorname{Re}(z)\log|w|-2\operatorname{Im}(z)\arg(w)\}
\leq B_K(1+|t|)^{N_K},
$$
because $\arg(w)\in(-\pi/2,\pi/2)$. Combining this bound with the assumed Gaussian decay gives an integrable
majorant, uniformly for $z\in K$. The same argument applies after multiplication by any fixed power of $\log w$. Hence the integral converges locally uniformly in $z$, and differentiation under the integral sign gives (\ref{eq:pmt-derivatives}). The resulting function is entire.
\end{proof}

For example, the Gaussian kernel $f(u)=e^u$ satisfies $|e^{(\sigma+it)^2}|=e^{\sigma^2-t^2}$, which is the convergence mechanism in the Gamma-function representation below.

The notation retains the dependence on $\sigma$. Independence of $\sigma$ is a contour-deformation property and is not automatic.

\begin{proposition}
\label{prop:pmt-contour-deformation}
Let $0<\sigma_0<\sigma_1$, and suppose that $f$ is holomorphic in a neighborhood of
$\{w^2:\sigma_0\leq\operatorname{Re}(w)\leq\sigma_1\}$. Assume also that for every compact $K\subset\mathbb{C}$ there exist constants $A,\eta>0$ and $N\geq0$ such that
$$
|w^{2z}f(w^2)|
\leq
A(1+|\operatorname{Im}(w)|)^N
\exp(-\eta|\operatorname{Im}(w)|^2)
$$
for all $z\in K$ and all $w$ with $\sigma_0\leq\operatorname{Re}(w)\leq\sigma_1$. Then
$$
\mathcal{P}_{\sigma_0}[f](z)=\mathcal{P}_{\sigma_1}[f](z),
\qquad z\in\mathbb{C}.
$$
\end{proposition}
\begin{proof}
For fixed $z$, apply Cauchy's theorem to $h_z(w)=w^{2z}f(w^2)$ on the rectangle
$$
\{w:\sigma_0\leq\operatorname{Re}(w)\leq\sigma_1,\ |\operatorname{Im}(w)|\leq T\}.
$$
The vertical sides give $i\mathcal{P}_{\sigma_1}[f](z)$ and $-i\mathcal{P}_{\sigma_0}[f](z)$ after truncation. The horizontal sides vanish as $T\to\infty$ by the assumed bound. Hence the two vertical-line integrals agree. The argument is locally uniform in $z$ on compact sets.
\end{proof}

For meromorphic kernels the same deformation produces residue terms. If $f$ is meromorphic in the strip above and has no poles on the boundary contours, then
$$
\mathcal{P}_{\sigma_1}[f](z)-\mathcal{P}_{\sigma_0}[f](z)
=
2\pi
\sum_{\sigma_0<\operatorname{Re}(w)<\sigma_1}
\operatorname{res}_{w}
\{w^{2z}f(w^2)\},
$$
where the sum is over the poles in the vertical strip. Equivalently, in the $u$-plane,
\begin{equation}
\mathcal{P}_{\sigma_1}[f](z)-\mathcal{P}_{\sigma_0}[f](z)
=
\pi
\sum_{\rho}
\operatorname{res}_{u=\rho}
\{u^{z-1/2}f(u)\},
\label{eq:pmt-residue-change-u}
\end{equation}
where the sum is over the poles crossed by the deformation from $C_{\sigma_0}$ to $C_{\sigma_1}$.

Thus entire kernels such as $f(u)=e^{\alpha u}$, with $\operatorname{Re}(\alpha)>0$, give $\sigma$-independent transforms. In contrast, the kernels
$$
\frac{e^{au}}{1-e^u},
\qquad
\frac{e^{au}}{1+e^u},
$$
are meromorphic. Their transforms are invariant only under deformations of $C_\sigma$ that do not cross poles; crossing a pole changes the value by the corresponding residue contribution in (\ref{eq:pmt-residue-change-u}).

One practical advantage of the parabolic parametrization is that, for the kernels considered here, the transform is represented by an ordinary vertical-line integral with $\sigma>0$ fixed. For entire kernels the value is independent of $\sigma$, while for meromorphic kernels it is independent of $\sigma$ within a fixed pole-free contour class. This avoids having to define the contour integral directly as a limiting Hankel keyhole contour integral.

\section{The Gaussian transform}
\label{sec:gaussian-transform}

Let
\begin{equation}
G(z)=\mathcal{P}_\sigma[e^u](z)
=
\int_{-\infty}^{\infty}w^{2z}e^{w^2}dt,
\qquad w=\sigma+it.
\label{eq:gaussian-pmt}
\end{equation}
By Proposition \ref{prop:pmt-contour-deformation}, the right-hand side is independent of $\sigma>0$. We use the shifted variable
$$
s=z+\tfrac{1}{2}.
$$

\begin{theorem}[Gaussian PMT]
\label{thm:gaussian-pmt}
For all $z\in\mathbb{C}$, with $s=z+\tfrac{1}{2}$,
\begin{equation}
G(z)=\frac{\pi}{\Gamma(1-s)}
=
\Gamma(s)\sin(\pi s).
\label{eq:G-s-form}
\end{equation}
Equivalently, for all $s\in\mathbb{C}$,
$\frac{1}{\Gamma(s)}=\frac{1}{\pi}G(\tfrac{1}{2}-s)$ and $\Gamma(s)\sin(\pi s)
=G(s-\tfrac{1}{2})$.
\end{theorem}
\begin{proof}
By (\ref{eq:pmt-contour-form}),
$G(z)=\frac{1}{2i}\int_{C_\sigma}e^u u^{z-1/2}du
=
\frac{1}{2i}\int_{C_\sigma}e^u u^{-(1-s)}du$. 
Deforming $C_\sigma$ to the standard Hankel contour and using the classical Hankel representation of $1/\Gamma$ \cite[5.9.2]{DLMF} gives $G(z)=\frac{\pi}{\Gamma(1-s)}$.
This deformation takes place in the slit plane $\mathbb{C}\setminus(-\infty,0]$.
For each fixed $\sigma>0$, the contour $C_\sigma$ has vertex at $\sigma^2>0$
and is homotopic in this slit plane to the standard Hankel contour about
$(-\infty,0]$. The Gaussian decay of $e^{w^2}$ along $w=\sigma+it$ provides
the tail control needed to justify the deformation of the unbounded contour.
Euler's reflection formula, \citet[5.5.3]{DLMF}, gives $\frac{\pi}{\Gamma(1-s)}=\Gamma(s)\sin(\pi s)$. 
The equivalent identities follow by replacing $z$ with $s-\tfrac{1}{2}$ or $\tfrac{1}{2}-s$.
\end{proof}

Theorem \ref{thm:gaussian-pmt} is a parabolic form of Hankel's reciprocal-gamma formula. Under the change of variables $u=w^2$, the PMT identity becomes a Hankel-type contour integral over $C_\sigma$, and deforming $C_\sigma$ to the standard Hankel contour gives the classical representation. Thus the result is not a new analytic continuation of $1/\Gamma(s)$, but a Gaussian-damped vertical-line parametrization of the same contour integral. It should also be compared with the classical vertical-line formula of Laplace,
$$
\frac{1}{\Gamma(s)}
=
\frac{1}{2\pi}
\int_{-\infty}^{\infty}(\sigma+it)^{-s}e^{\sigma+it}dt,
\qquad \sigma>0,
$$
valid for $\operatorname{Re}(s)>0$. This formula goes back to \citet[p. 264]{Laplace:1785} and \citet[p. 134]{Laplace:1812}; see also \citet[Section 63.8]{Nielsen:1906}, \citet{Pribitkin:2002}, and \citet[eq. 8.315.2]{GradshteynRyzhik:2007}.\footnote{The denominator in \citet[eq. 8.315.2]{GradshteynRyzhik:2007} has a typographical error: printed as $(a+it)^2$ rather than $(a+it)^z$.} In contrast with the Laplace integral, which relies on oscillation and is only conditionally convergent in part of its natural domain, the Gaussian PMT representation contains the factor $e^{w^2}$. Along $w=\sigma+it$,
$|e^{w^2}|=e^{\sigma^2-t^2}$, giving absolute convergence for all $z\in\mathbb{C}$. This Gaussian damping is the feature used below when Dirichlet-type kernels are combined with the Gaussian PMT.

The identity in Theorem \ref{thm:gaussian-pmt} immediately gives the zero set
and logarithmic derivative of $G$.

\begin{corollary}
\label{cor:G-consequences}
The function $G$ is entire and has simple zeros at
$z=\tfrac{1}{2}+n$, $n\in\mathbb{N}_0$. Away from these zeros,
\begin{equation}
\frac{G^\prime(z)}{G(z)}=\psi(1-s),
\label{eq:G-log-derivative}
\end{equation}
where $\psi=\Gamma^\prime/\Gamma$ is the digamma function. In particular,
$$
\int_{-\infty}^{\infty}\log(\sigma+it)e^{(\sigma+it)^2}dt
=
-\frac{\sqrt{\pi}}{2}(\gamma+2\log2),
$$
where $\gamma$ is Euler's constant.
\end{corollary}

\begin{proof}
The zeros follow from $G(z)=\pi/\Gamma(1-s)$ and the simple poles of $\Gamma$
at the nonpositive integers. Logarithmic differentiation gives
(\ref{eq:G-log-derivative}). On the other hand, differentiating
(\ref{eq:gaussian-pmt}) under the integral sign gives
$$
G^\prime(z)=
\int_{-\infty}^{\infty}
2\log(w)w^{2z}e^{w^2}dt.
$$
At $z=0$, $G(0)=\sqrt{\pi}$ and
$\psi(\tfrac{1}{2})=-\gamma-2\log2$, which gives the last expression.
\end{proof}

By Theorem \ref{thm:gaussian-pmt} and Euler's reflection formula we have $G(z)G(-z)=\pi\cos(\pi z)$ and taking logarithmic derivatives gives
$$
\frac{G^\prime(z)}{G(z)}
-
\frac{G^\prime(-z)}{G(-z)}
=
-\pi\tan(\pi z),
\qquad z\notin \tfrac{1}{2}+\mathbb{Z}.
$$
With $s=z+\tfrac{1}{2}$, this is the standard digamma reflection formula $\psi(1-s)-\psi(s)=\pi\cot(\pi s)$ written in the $z$-coordinate, see \citet[5.5.4]{DLMF}.

\section{Dirichlet composition and zeta-type functions}
\label{sec:dirichlet-zeta}

We next combine the Gaussian PMT with elementary exponential generating
functions. We continue to write $s=z+\tfrac{1}{2}$. For $\alpha>0$, define
$$
G(z,\alpha)
=\int_{-\infty}^{\infty}w^{2z}e^{\alpha w^2}dt,
\qquad w=\sigma+it.
$$
Thus $G(z)=G(z,1)$ and by Proposition \ref{prop:pmt-contour-deformation}, the right-hand side is
independent of $\sigma>0$.

\begin{proposition}[Scaling]
\label{prop:G-scaling}
For all $z\in\mathbb{C}$ and $\alpha>0$,
$$
G(z,\alpha)=\alpha^{-(z+1/2)}G(z)=\alpha^{-s}G(z).
$$
\end{proposition}

\begin{proof}
Set $\tilde{w}=\sqrt{\alpha}w$. Then
$\tilde{w}=\tilde{\sigma}+i\tilde{t}$, with
$\tilde{\sigma}=\sqrt{\alpha}\sigma$ and
$d\tilde{t}=\sqrt{\alpha}dt$. Hence
$w^{2z}=\alpha^{-z}\tilde{w}^{2z}$,
$e^{\alpha w^2}=e^{\tilde{w}^2}$, and
$dt=\alpha^{-1/2}d\tilde{t}$.
Therefore
$$
G(z,\alpha)
=
\alpha^{-(z+1/2)}
\int_{-\infty}^{\infty}
\tilde{w}^{2z}e^{\tilde{w}^2}d\tilde{t}.
$$
Since $\tilde{\sigma}>0$ and the Gaussian transform is independent of the
vertical-line parameter, the last integral equals $G(z)$. This completes the proof.
\end{proof}

Taking $\alpha=n+a$, where $a>0$ and $n\in\mathbb{N}_0$, gives
$G(z,n+a)=(n+a)^{-s}G(z)$. 
Equivalently,
$(n+a)^{-s} =
G(z,n+a)/G(z)$, away from the zeros of $G(z)$. At those zeros the quotient is
understood by removable continuation. Thus each Dirichlet factor $(n+a)^{-s}$
may be written as a ratio of two scaled Gaussian transforms.

For $a>0$, define
\begin{equation}
R(z,a)
=
\int_{-\infty}^{\infty}
w^{2z}\frac{e^{aw^2}}{1-e^{w^2}}dt,
\qquad 0<\sigma<\sqrt{\pi},
\label{eq:R-definition}
\end{equation}
and
\begin{equation}
D(z,a)
=
\int_{-\infty}^{\infty}
w^{2z}\frac{e^{aw^2}}{1+e^{w^2}}dt,
\qquad 0<\sigma<\sqrt{\pi/2}.
\label{eq:D-definition}
\end{equation}
The restrictions on $\sigma$ place the parabolic contour in the pole-free
homotopy class between the origin and the first nonzero poles of the
corresponding meromorphic kernels. For such $\sigma$, the integrals define
entire functions of $z$.

We write
$$
\eta(s,a)=\sum_{n=0}^{\infty}(-1)^n(n+a)^{-s},
$$
initially in its half-plane of convergence and then by analytic continuation.
For $a=1$, this is the classical Dirichlet eta function.

\begin{theorem}[Parabolic zeta representations]
\label{thm:parabolic-zeta-representations}
For $a>0$ and $s=z+\tfrac{1}{2}$,
\begin{equation}
R(z,a)=G(z)\zeta(s,a),
\qquad
D(z,a)=G(z)\eta(s,a),
\label{eq:pmt-zeta-factorization}
\end{equation}
where the identities are understood meromorphically in $s$. Equivalently,
away from the zeros of $G(z)$,
$$
\zeta(s,a)=\frac{R(z,a)}{G(z)},
\qquad
\eta(s,a)=\frac{D(z,a)}{G(z)}.
$$
At zeros of $G(z)$, the quotients are interpreted by removable continuation
except where the corresponding zeta-type function has a genuine pole.
\end{theorem}

\begin{proof}
We first prove the identities for $\operatorname{Re}(s)>1$. By
(\ref{eq:pmt-contour-form}),
$$
R(z,a)
=
\frac{1}{2i}
\int_{C_\sigma}
u^{s-1}\frac{e^{au}}{1-e^u}du,
\qquad s=z+\tfrac{1}{2}.
$$
Deform $C_\sigma$ to the standard Hankel contour $H$ about $(-\infty,0]$, without crossing the nonzero poles of $(1-e^u)^{-1}$. This is the classical Hankel-type contour used in the contour representation of the Hurwitz zeta function; see, for example, \citet[Section 13.13, pp. 277--278]{WhittakerWatson:2021}.
For $\operatorname{Re}(s)>1$, the contribution from the small circle around the origin vanishes, since
$u^{s-1}e^{au}/(1-e^u) = O(u^{s-2})$, as $u\to0$. Thus
$$
R(z,a)
= \frac{1}{2i}
\int_H
u^{s-1}\frac{e^{au}}{1-e^u}du.
$$
Evaluating the jump across the negative real axis gives
$$
R(z,a) = \sin(\pi s) \int_0^\infty x^{s-1}\frac{e^{-ax}}{1-e^{-x}}dx.
$$
For $\operatorname{Re}(s)>1$, the classical Mellin representation of the
Hurwitz zeta function gives
$$
\int_0^\infty
x^{s-1}\frac{e^{-ax}}{1-e^{-x}}dx
=
\Gamma(s)\zeta(s,a);
$$
see, for example, \citet[Section 12.1]{Apostol:1976}. Therefore
$R(z,a)=\Gamma(s)\sin(\pi s)\zeta(s,a)
=G(z)\zeta(s,a)$,  where the last equality follows from Theorem \ref{thm:gaussian-pmt}.

The proof for $D$ is analogous. By (\ref{eq:pmt-contour-form}),
$$
D(z,a)
=
\frac{1}{2i}
\int_{C_\sigma}
u^{s-1}\frac{e^{au}}{1+e^u}du.
$$
Deforming $C_\sigma$ to the standard Hankel contour $H$ without crossing poles
gives
$$
D(z,a)
=
\frac{1}{2i}
\int_H
u^{s-1}\frac{e^{au}}{1+e^u}du.
$$
The jump across the cut gives
$$
D(z,a)
=
\sin(\pi s)
\int_0^\infty
x^{s-1}\frac{e^{-ax}}{1+e^{-x}}dx.
$$
For $\operatorname{Re}(s)>1$,
$\int_0^\infty
x^{s-1}\frac{e^{-ax}}{1+e^{-x}}dx
=\Gamma(s)\eta(s,a)$. 
Hence $D(z,a)=G(z)\eta(s,a)$ on the open half-plane $\operatorname{Re}(s)>1$.

Both sides of the identities
$R(z,a)=G(z)\zeta(s,a)$ and $D(z,a)=G(z)\eta(s,a)$ are meromorphic functions of $s$, equivalently of $z=s-\tfrac{1}{2}$. 
In fact, the apparent pole of $G(z)\zeta(s,a)$ at $s=1$ is removable, because $G(z)=\Gamma(s)\sin(\pi s)$ has a simple zero there. Since the identities agree on the open half-plane $\operatorname{Re}(s)>1$, they extend by the identity theorem for meromorphic functions. The quotient forms follow wherever
$G(z)\neq0$, with removable continuation at zeros of $G(z)$ except where the
corresponding zeta-type function has a genuine pole.
\end{proof}

For $a=1$, set $R(z)=R(z,1)$ and $D(z)=D(z,1)$. Then
$$
R(z)=G(z)\zeta(s),
\qquad
D(z)=G(z)\eta(s),
\qquad s=z+\tfrac{1}{2}.
$$
These are the PMT representations of the Riemann zeta function and the
Dirichlet eta function.

The formulas above are structurally equivalent to the classical Hankel-contour
representations of the Hurwitz zeta function. Indeed, the contour identity in
the proof, together with $G(z)=\Gamma(s)\sin(\pi s)$, gives
$$
\zeta(s,a)
=
\frac{\Gamma(1-s)}{2\pi i}
\int_{C_\sigma}
u^{s-1}\frac{e^{au}}{1-e^u}du,
$$
where the equality is understood by meromorphic continuation. Similarly,
\begin{equation}
\eta(s,a)
=
\frac{\Gamma(1-s)}{2\pi i}
\int_{C_\sigma}
u^{s-1}\frac{e^{au}}{1+e^u}du.
\label{eq:eta-parabolic-hankel-form}
\end{equation}
Thus the PMT does not provide a fundamentally new analytic continuation of
$\zeta(s,a)$. Rather, after the substitution $u=w^2$, it gives a parabolic
parametrization of the corresponding Hankel-type integrals. The new feature is
the explicit vertical-line representation in the $w$-plane, where
$\operatorname{Re}(w^2)=\sigma^2-t^2$ supplies Gaussian endpoint decay. This
places the Gamma, Hurwitz zeta, Riemann zeta, and Dirichlet eta functions in a
common transform notation, which is the basis for the dictionary developed
below.

\section{Critical-strip reformulations}
\label{sec:critical-strip}

The identities in Section \ref{sec:dirichlet-zeta} translate standard questions
about $\zeta(s)$ into equivalent statements about the parabolic-contour
integrals $R(z)$ and $D(z)$. We use the standard formulations of the Riemann
hypothesis and the Lindel\"of hypothesis; see, for example,
\citet{Titchmarsh:1986}. These reformulations do not, by themselves, give
new zero-free regions or growth estimates. Their role is to express the
critical-strip problem in the PMT coordinates $z=s-\tfrac{1}{2}$, where the
critical line becomes the imaginary axis.

\subsection{Zero correspondence}
\label{subsec:zero-correspondence}

Recall that $R(z)=G(z)\zeta(z+\tfrac{1}{2})$.
By Corollary \ref{cor:G-consequences}, the zeros of $G$ are
$$
z=\tfrac{1}{2}+n,\qquad n\in\mathbb{N}_{0}.
$$
Thus $G$ has no zeros in the open strip $|\operatorname{Re}(z)|<\tfrac{1}{2}$.

\begin{corollary}[Zero correspondence]
\label{cor:R-zero-correspondence}
For $|\operatorname{Re}(z)|<\tfrac{1}{2}$,
\begin{equation}
\zeta(z+\tfrac{1}{2})=0
\quad\Leftrightarrow\quad
R(z)=0.
\label{eq:R-zeta-zero-correspondence}
\end{equation}
The correspondence preserves multiplicities.
\end{corollary}

\begin{proof}
In the strip $|\operatorname{Re}(z)|<\tfrac{1}{2}$, the factor $G(z)$ is nonzero. Hence the zeros of $R(z)$ in this strip
are exactly the zeros of $\zeta(z+\tfrac{1}{2})$, with the same multiplicities.
\end{proof}

It follows immediately that the Riemann hypothesis is equivalent to the assertion that every zero of $R(z)$ in the open strip $|\operatorname{Re}(z)|<\tfrac{1}{2}$ lies on the imaginary axis. This is only the usual critical-strip formulation written in the shifted coordinate $z=s-\tfrac{1}{2}$; its advantage is that the critical strip is centered at zero and the critical line becomes $\operatorname{Re}(z)=0$.

\subsection{Lindel\"of reformulation}
\label{subsec:lindelof-reformulation}

On the critical line, Theorem \ref{thm:parabolic-zeta-representations} gives
$\zeta(\tfrac{1}{2}+i\tau)=\frac{R(i\tau)}{G(i\tau)}$,
 $\tau\in\mathbb{R}$, where the denominator is explicit, by Theorem \ref{thm:gaussian-pmt},
$
G(i\tau)=\cosh(\pi\tau)\Gamma(\tfrac{1}{2}+i\tau)$. 
Using the standard identity $|\Gamma(\tfrac{1}{2}+i\tau)|^2=\pi/\cosh(\pi\tau)$
\cite[5.4.4]{DLMF}, it follows that
$$
|G(i\tau)|=\sqrt{\pi\cosh(\pi\tau)}
\sim \sqrt{\frac{\pi}{2}}e^{\pi|\tau|/2},
\qquad\text{as}\quad |\tau|\to\infty.
$$

\begin{corollary}[Lindel\"of in PMT form]
\label{cor:lindelof-pmt}
The Lindel\"of hypothesis is equivalent to the assertion that, for every
$\epsilon>0$,
$$
R(i\tau)
=
O_\epsilon\left(e^{\pi|\tau|/2}(1+|\tau|)^\epsilon\right),
\qquad |\tau|\to\infty.
$$
Equivalently, for any admissible $0<\sigma<\sqrt{\pi}$,
$$
\left|
\int_{-\infty}^{\infty}
(\sigma+it)^{2i\tau}
\frac{1}{e^{-(\sigma+it)^2}-1}dt
\right|
=
O_\epsilon\left(e^{\pi|\tau|/2}(1+|\tau|)^\epsilon\right).
$$
\end{corollary}

\begin{proof}
The identities above imply
$$
|\zeta(\tfrac{1}{2}+i\tau)|
=
\frac{|R(i\tau)|}{|G(i\tau)|}
\asymp
e^{-\pi|\tau|/2}|R(i\tau)|.
$$
Thus $|\zeta(\tfrac{1}{2}+i\tau)|=O_\epsilon((1+|\tau|)^\epsilon)$
for every $\epsilon>0$ if and only if the stated growth condition for
$R(i\tau)$ holds. The integral form follows from
$$
R(i\tau)
=
\int_{-\infty}^{\infty}
w^{2i\tau}\frac{e^{w^2}}{1-e^{w^2}}dt
=
\int_{-\infty}^{\infty}
w^{2i\tau}\frac{1}{e^{-w^2}-1}dt,
\qquad w=\sigma+it.
$$
\end{proof}
This reformulation isolates the explicit exponential factor in $G(i\tau)$, but
it does not remove the oscillatory difficulty of estimating the remaining
parabolic integral $R(i\tau)$.

The same factorization also has a boundary interpretation. The parabolic
integral $R(z)$ is entire and remains well defined on the boundary lines
$|\operatorname{Re}(z)|=\tfrac{1}{2}$. On the right boundary, for
$\tau\neq0$,
$$
\zeta(1+i\tau)=\frac{R(\tfrac{1}{2}+i\tau)}
{G(\tfrac{1}{2}+i\tau)}.
$$
At $\tau=0$, the zero $G(\tfrac{1}{2})=0$ accounts for the pole of $\zeta$ at
$s=1$; indeed $R(\tfrac{1}{2})=-\pi$. Thus the PMT factorization extends to
the boundary of the critical strip, with the singularity at $s=1$ isolated
entirely in the explicit factor $G$.
\medskip

\begin{table}[hbt]
\centering
\small
\renewcommand{\arraystretch}{1.35}
\begin{tabularx}{\textwidth}{@{}p{3.4cm}p{1.9cm}p{3.9cm}X@{}}
\toprule
\textbf{Kernel} & $f(u)$ & $\mathcal{P}_{\sigma}[f](z)$ & \textbf{Derivation / status} \\
\midrule

Gaussian
& $e^u$
& $G(z)$
& Definition; Theorem \ref{thm:gaussian-pmt}. \\[2mm]

Scaled Gaussian
& $e^{\alpha u}$
& $\alpha^{-s}G(z)$
& Scaling, $\alpha>0$; Proposition \ref{prop:G-scaling}. \\[3mm]

Monomial Gaussian
& $u^k e^u$
& $G(z+k)$
& Monomial shift, $k\in\mathbb{N}_0$. \\[2mm]

Log-Gaussian
& $e^u\log u$
& $G^{\prime}(z)$
& Differentiation in $z$. \\[2mm]

Hurwitz
& $\displaystyle\frac{e^{au}}{1-e^u}$
& $R(z,a)=G(z)\zeta(s,a)$
& Hankel deformation; $a>0$, $0<\sigma<\sqrt{\pi}$. \\[3mm]

Riemann
& $\displaystyle\frac{e^u}{1-e^u}$
& $R(z)=G(z)\zeta(s)$
& Hurwitz row with $a=1$. \\[3mm]

Alternating Hurwitz
& $\displaystyle\frac{e^{au}}{1+e^u}$
& $D(z,a)=G(z)\eta(s,a)$
& Hankel deformation; $a>0$, $0<\sigma<\sqrt{\pi/2}$. \\[3mm]

Dirichlet eta
& $\displaystyle\frac{e^u}{1+e^u}$
& $D(z)=G(z)\eta(s)$
& Alternating Hurwitz row with $a=1$. \\

\bottomrule
\end{tabularx}
\caption{Dictionary of the main PMT identities. The quotient identities for
zeta-type functions are interpreted meromorphically; apparent singularities at
zeros of $G$ are removable except where the zeta-type function has a genuine
pole.}
\label{tab:pmt-dictionary}
\end{table}

\section{A PMT dictionary}
\label{sec:pmt-dictionary}

We collect the main transforms derived above. Table~\ref{tab:pmt-dictionary} is intended as a dictionary of proved identities, not as an independent extension principle.
The elementary operations used in the derivations are linearity, monomial shifts, differentiation in $z$, scaling in the Gaussian kernel, and contour deformation to Hankel-type contours. In particular, the monomial Gaussian row
follows from $\mathcal{P}_\sigma[u^k e^u](z)=G(z+k)$, while the logarithmic Gaussian row follows from $\mathcal{P}_\sigma[e^u\log u](z)=G^\prime(z)$, using
$\log(w^2)=2\log w$ on the chosen branch and differentiation under the integral sign. Scaling and deformation identities are to be interpreted with the same contour restrictions as in Proposition \ref{prop:pmt-contour-deformation}.
Throughout the table, $s=z+\tfrac{1}{2}$. The restrictions on $\sigma$ specify the pole-free contour class.

\section{Conclusion}
\label{sec:conclusion}

We introduced the Parabolic Mellin Transform as a Mellin-type contour transform obtained from the quadratic parametrization $u=w^2$. The main analytic feature of this parametrization is that exponential kernels on the parabolic contour become Gaussian-damped vertical-line integrals. This gives a compact contour-theoretic route to the Gaussian transform $G$, which represents both the Gamma function and its reciprocal at reflected arguments.

The same Gaussian transform provides a common factor in the parabolic representations of Hurwitz zeta, Riemann zeta, and Dirichlet eta functions. The scaling identity for $G(z,\alpha)$ turns Gaussian transforms into Dirichlet factors, while Hankel deformation gives the corresponding zeta-type identities. In this sense, the PMT is best viewed as a transform dictionary: it reorganizes classical Hankel-contour formulas in a vertical-line parametrization with explicit Gaussian endpoint decay.

The critical-strip consequences are therefore reformulations rather than new zero-free regions or growth estimates. The factorization $R(z)=G(z)\zeta(z+\tfrac{1}{2})$ gives an exact zero correspondence in the open strip $|\operatorname{Re}(z)|<\tfrac{1}{2}$, and the explicit size of $G(i\tau)$ gives an equivalent PMT form of the Lindel\"of hypothesis. These identities suggest that parabolic-contour representations may be a useful way to compare Gamma-type normalizations, zeta-type kernels, and critical-line growth within a single analytic notation.

\section*{Acknowledgments}
We thank Christian Berg (University of Copenhagen), Steen Thorbj{\o}rnsen (Aarhus University), and Boris Hanin (Princeton University) for their helpful comments and suggestions.

\section*{Funding}
Chen Tong acknowledges financial support from the National Natural Science Foundation of China (72301227) and the Fujian Provincial Natural Science Foundation of China (2025J08008).

{\footnotesize\bibliographystyle{apalike}
\bibliography{prh}
}{\footnotesize\par}

\newpage

\appendix
\setcounter{section}{0}\renewcommand{\thesection}{S.\arabic{section}}
\setcounter{equation}{0}\renewcommand{\theequation}{S.\arabic{equation}}
\setcounter{table}{0}\renewcommand{\thetable}{S.\arabic{table}}
\setcounter{lemma}{0}\renewcommand{\thelemma}{S.\arabic{lemma}}

\part*{Supplement}

\section{An auxiliary sufficient condition}
\label{subsec:symmetric-sufficient-condition}

The zero correspondence in Corollary \ref{cor:R-zero-correspondence} can be
embedded in a symmetric auxiliary function. Define
$$
\mathcal{S}(z)
=
\frac{1}{2}\{R(z)-R(-z)+D(z)-D(-z)\}.
$$
If $R(z)=0$ and $|\operatorname{Re}(z)|<\tfrac{1}{2}$, then
$\zeta(z+\tfrac{1}{2})=0$. By the functional equation,
$\zeta(\tfrac{1}{2}-z)=0$, and hence $R(-z)=0$. Since
$D(z)=G(z)\eta(z+\tfrac{1}{2})$ and
$\eta(s)=(1-2^{1-s})\zeta(s)$, it follows that $D(z)=D(-z)=0$ as well.
Thus every zero of $R$ in the open critical strip is also a zero of
$\mathcal{S}$.

Let $u_t=(\sigma+it)^2$ with $0<\sigma<\sqrt{\pi/2}$ and define
\begin{equation}
\mathcal{X}(\tau)
=
\int_{-\infty}^{\infty}
\frac{\sin(\tau\log u_t)}{\sinh(u_t)}dt.
\label{eq:X-definition}
\end{equation}
The restriction on $\sigma$ keeps the contour in the pole-free homotopy class
for both $R$ and $D$.

For fixed $\tau$ with $|\operatorname{Im}(\tau)|<\tfrac{1}{2}$, the integral
defining $\mathcal{X}(\tau)$ converges absolutely. Indeed,
$\log u_t=O(\log|t|)+iO(1)$ as $|t|\to\infty$, so
$\sin(\tau\log u_t)$ grows at most polynomially in $|t|$, while
$1/\sinh(u_t)=O(e^{-t^2})$ along $u_t=(\sigma+it)^2$.

\begin{proposition}[Auxiliary sufficient condition for RH]
\label{prop:X-sufficient-RH}
If every zero of $\mathcal{X}(\tau)$ in the strip
$|\operatorname{Im}(\tau)|<\tfrac{1}{2}$ is real, then the Riemann hypothesis
holds.
\end{proposition}

\begin{proof}
Using
$$
\frac{1}{e^{-u}-1}+\frac{1}{e^{-u}+1}
=
-\frac{1}{\sinh(u)},
$$
we obtain, with $u=u_t$,
$$
\mathcal{S}(z)
=
-\frac{1}{2}
\int_{-\infty}^{\infty}
\frac{u_t^z-u_t^{-z}}{\sinh(u_t)}dt
=
-\int_{-\infty}^{\infty}
\frac{\sinh(z\log u_t)}{\sinh(u_t)}dt.
$$
Therefore $\mathcal{S}(i\tau)=-i\mathcal{X}(\tau)$.

Let $\rho=\tfrac{1}{2}+z$ be a nontrivial zero of $\zeta$, so
$|\operatorname{Re}(z)|<\tfrac{1}{2}$. By the preceding paragraph,
$\mathcal{S}(z)=0$. Writing $z=i\tau$, equivalently $\tau=-iz$, gives
$|\operatorname{Im}(\tau)|<\tfrac{1}{2}$ and
$\mathcal{X}(\tau)=0$ since $\mathcal{S}(i\tau)=-i\mathcal{X}(\tau)$. By hypothesis, this zero
$\tau$ is real. Hence $z=i\tau$ is purely imaginary, so
$\operatorname{Re}(\rho)=\tfrac{1}{2}$. This proves the Riemann hypothesis (RH).
\end{proof}

The converse is not asserted. The function $\mathcal{X}$ has zeros that do not
come from zeros of $\zeta$; for instance $\mathcal{X}(0)=0$ by construction.
Thus Proposition \ref{prop:X-sufficient-RH} gives only an auxiliary sufficient
condition for RH, not an equivalent reformulation.

\section{A probabilistic derivation of the Gaussian PMT}
\label{app:probabilistic-gaussian-pmt}

The proof of Theorem \ref{thm:gaussian-pmt} in the main text is contour-theoretic. This appendix records a short probabilistic route to the same identity, which was the original source of the Gaussian PMT.

Let $X\sim N(0,1)$ and let $w=\sigma+it$, with $\sigma>0$. For real $r>0$, the vertical-line representation of absolute powers gives
\begin{equation}
|x|^r
=
\frac{\Gamma(r+1)}{2\pi}
\int_{-\infty}^{\infty}
\frac{e^{wx}+e^{-wx}}{w^{r+1}}dt,
\qquad x\in\mathbb{R}.
\label{eq:absolute-power-origin}
\end{equation}
This is the symmetrized Bromwich inversion formula for the power function. Taking expectations in (\ref{eq:absolute-power-origin}) and using
$$
\mathbb{E}e^{wX}=\mathbb{E}e^{-wX}=e^{w^2/2}
$$
gives
\begin{equation}
\mathbb{E}|X|^r
=
\frac{\Gamma(r+1)}{\pi}
\int_{-\infty}^{\infty}
w^{-(r+1)}e^{w^2/2}dt.
\label{eq:normal-absolute-origin}
\end{equation}
Combining this with the well-known expression,
$
\mathbb{E}|X|^r
=
\frac{2^{r/2}}{\sqrt{\pi}}
\Gamma\left(\frac{r+1}{2}\right)
$ and using Legendre's duplication formula yields
\begin{equation}
\int_{-\infty}^{\infty}
w^{-(r+1)}e^{w^2/2}dt
=
\frac{\pi 2^{-r/2}}{\Gamma(r/2+1)}.
\label{eq:gaussian-origin-r}
\end{equation}

Set $s=r/2+1$. Then $r=2s-2$, and (\ref{eq:gaussian-origin-r}) becomes
\begin{equation}
\int_{-\infty}^{\infty}
w^{1-2s}e^{w^2/2}dt
=
\frac{\pi 2^{1-s}}{\Gamma(s)},
\qquad s>1.
\label{eq:gaussian-origin-s}
\end{equation}
Finally, substitute $w=\sqrt{2}\tilde{w}$, equivalently
$\tilde{w}=w/\sqrt{2}$. Since $dt=\sqrt{2}d\tilde{t}$,
\[
\int_{-\infty}^{\infty}
w^{1-2s}e^{w^2/2}dt
=
2^{1-s}
\int_{-\infty}^{\infty}
\tilde{w}^{1-2s}e^{\tilde{w}^2}d\tilde{t}.
\]
Thus
$$
\int_{-\infty}^{\infty}
\tilde{w}^{1-2s}e^{\tilde{w}^2}d\tilde{t}
=
\frac{\pi}{\Gamma(s)}.
$$
Since $G(z)=\int_{-\infty}^{\infty}w^{2z}e^{w^2}dt$, this is precisely
\begin{equation}
G\left(\tfrac{1}{2}-s\right)
=
\frac{\pi}{\Gamma(s)}.
\label{eq:probabilistic-G-reciprocal}
\end{equation}
Both sides are entire functions of $s$, so the identity extends from $s>1$ to all $s\in\mathbb{C}$.

\section{Numerical checks}

The identities in the main text are exact analytic identities. This section records a few numerical evaluations of the parabolic integrals as consistency checks. The computations also illustrate the independence of the vertical-line
parameter $\sigma$ within the admissible contour class.
The first table checks the Gaussian identity
$\frac{1}{\pi}G(\tfrac{1}{2}-s)=1/\Gamma(s)$.
The second table checks the zeta identity
$R(s-\tfrac{1}{2})/G(s-\tfrac{1}{2})=\zeta(s)$.

\begin{center}
\begin{small}
\begin{tabular*}{\textwidth}{@{\extracolsep{\fill}}llll@{ }}
        \toprule
        $s$ & Exact / Known Value & Computed Value & Character / Remarks \\
        \midrule
        $1$ & $1$ & $1.00000$ & Clean check, $G(-1/2) = {\pi}$ \\
        $1/2$ & $1/\sqrt{\pi} \approx 0.56419$ & $0.56419$ & $z=0$, pure Gaussian integral \\
        $3/2$ & $2/\sqrt{\pi} \approx 1.12838$ & $1.12838$ & $G(-1) = 2\sqrt{\pi}$ \\
        $-1/2$ & $-1/(2\sqrt{\pi}) \approx -0.28209$ & $-0.28209$ & Tests a negative value \\
        $3/4$ & $\approx 0.81605$ & $0.81605$ & No closed form; genuine numerical test \\
        $1 + i$ & $\approx 1.83074 + 0.56961i$ & $1.83074 + 0.56961i$ & Tests complex arguments \\
        \midrule
        \multicolumn{4}{l}{\textbf{$\sigma$-Independence Demonstration} (Evaluated at $s=3/4$)} \\
        \midrule
        $3/4$ & Fixed across $\sigma \in \{0.5, 1.0, 1.5, 2.0\}$ & $0.81604893909826$ & Demonstrates $\sigma$-invariance \\
        \bottomrule
        \end{tabular*}

\begin{tabular*}{\textwidth}{@{\extracolsep{\fill}}llll@{ }}
        \toprule
        $s$ & Exact / Known Value & Computed Value & Character / Remarks \\
        \midrule
        $3/2$ & $\approx 2.61238$ & $2.61238$ & Real, $\operatorname{Re}(s)>1$; convergent cross-check \\
        $1/2$ & $\approx -1.46035$ & $-1.46035$ & Uses the single-integral formula \\
        $1/2 + 5i$ & $\approx 0.70181 + 0.23104i$ & $0.70181 + 0.23104i$ & Critical line; inside the strip \\
        $1/2 + 14.1347i$ & $0$ (First nontrivial zero) & $\approx 0$ ($| \text{res} | < 10^{-14}$) & Dramatic; integral clearly vanishes \\
        \midrule
        \multicolumn{4}{l}{\textbf{$\sigma$-Independence Demonstration} (Evaluated at $s=3/2$)} \\
        \midrule
        $3/2$ & Fixed across $\sigma \in (0, \sqrt{\pi})$ & $2.612375348685$ & Evaluated at $\sigma \in \{0.5, 1.0, 1.5, 1.7\}$ \\
        \bottomrule
    \end{tabular*}
\end{small}
\end{center}

\end{document}